\documentclass[12pt]{article}
\usepackage{graphicx}
\usepackage[margin=30truemm]{geometry}

\usepackage{amsmath}
\usepackage{amsthm}

\newtheorem{thm}{Theorem}[section]
\newtheorem{lem}[thm]{Lemma}

\newtheorem{prop}[thm]{Property}
\newtheorem{conj}[thm]{Conjecture}

\usepackage{setspace}
\title{Completely Independent Spanning Trees in $k$-Outerplanar Triangulated Discs}
\author{Toru Araki}
\date{Graduate School of Informatics, Gunma University \\ Maebashi, Gunma, Japan}

\begin{document}

\maketitle

\begin{abstract}
  Let $T_{1}, T_{2}, \dots, T_{k}$ be $k$ spanning trees of a graph $G$.
  For any pair of vertices $u$ and $v$, if the $u$--$v$ paths in the $k$ spanning trees are pairwise openly disjoint, then the spanning trees are called completely independent spanning trees (CISTs) of $G$.
  In this paper, we first prove that every 3-connected 2-outerplanar triangulated disc has two completely independent spanning trees.
  Next, for a 3-connected 3-outerplanar triangulated disc $G$, we provide sufficient conditions for $G$ to have two completely independent spanning trees.
  We provide an example of a 3-connected 4-outerplanar triangulation that does not have two completely independent spanning trees.
\end{abstract}


\section{Introduction}
\label{sec:introduction}

The graphs considered in this paper are undirected, finite, and simple.
For a graph $G$, the sets $V(G)$ and $E(G)$ are the vertex and edge sets of $G$, respectively.
The set of vertices adjacent to $v$ is the \emph{neighborhood} of $v$, denoted by $N_{G}(v)$.
For a set $U \subseteq V(G)$, we define $N_{G}(U) = \bigcup_{v \in U} N_{G}(v)$.
The degree of $v$ in $G$ is $\deg_{G} v = |N_{G}(v)|$.
For $S \subseteq V(G)$, $G-S$ denotes the graph obtained from $G$ by deleting all the vertices in $S$ and the edges incident with them.
For a nonempty subset $S \subseteq V(G)$, the subgraph induced by $S$ is denoted by $G[S]$.
A walk and a path from $u$ to $v$ are referred to as a $u$--$v$ walk and a $u$--$v$ path, respectively.
Two $u$--$v$ paths $P$ and $Q$ are \emph{openly disjoint} if they are edge-disjoint and have no common vertex except for $u$ and $v$.
Let $T_{1}, T_{2}, \dots, T_{k}$ be $k$ spanning trees of $G$.
If, for any pair of vertices $u$ and $v$, the $u$--$v$ paths in $T_{1}, T_{2}, \dots, T_{k}$ are pairwise openly disjoint, then the spanning trees are called \emph{completely independent spanning trees} (CISTs) of $G$.

A \emph{planar embedding} of a planar graph $G$ is an embedding of $G$ in a plane such that the edges of $G$ do not cross each other.
A graph is \emph{planar} if it has a planar embedding.
A planar graph with a planar embedding is called a \emph{plane graph}.
The unbounded face of a plane graph is called the \emph{outer face}, whereas the other faces are called the \emph{inner faces}.
A face with three vertices is called a \emph{triangle}.
A \emph{planar triangulation} is a plane graph such that all faces (including the outer face) are triangles.
A \emph{triangulated disc} is a plane graph with a plane embedding such that all inner faces are triangles.

A planar embedding of a graph is \emph{$1$-outerplanar} if all vertices lie on the outer face.
For $k \geq 2$, a planar embedding of a graph is \emph{$k$-outerplanar} if removing all the vertices on the outer face results in a $(k-1)$-outerplanar embedding of the graph.
A plane graph is $k$-outerplanar if it has a $k$-outerplanar embedding.
A 1-outerplanar graph is called an \emph{outerplanar graph}.
The \emph{outerplanarity} of a planar graph is the minimum $k$ such that it has a $k$-outerplanar embedding.

The concept of completely independent spanning trees was introduced by Hasunuma~\cite{hasunuma01}.
Completely independent spanning trees provide multiple disjoint paths between any pair of nodes in the network; hence, the concept can be applied to fault-tolerant broadcasting in networks.
Readers are referred to~\cite{CHANG202175,9217977,pai21,pai22} for CISTs and their applications.
The existence of CISTs in special graph classes has been studied, for example, the underlying graph of line digraphs~\cite{hasunuma01}, 4-connected planar triangulation~\cite{hasunuma02}, torus graphs~\cite{hasunuma11}, complete graphs~\cite{pai13}, $k$-trees~\cite{araki15:_compl}, split graphs~\cite{chen23:_two}, $P_{4}$-free graphs~\cite{yuan25:_p}, and line graphs~\cite{hasunuma23:_compl}.
Recently, several sufficient conditions have been investigated to guarantee the existence of CISTs in graphs.
The author proved that Dirac’s and Fleischner's conditions are sufficient for the existence of CISTs~\cite{araki14}.
These conditions are inspired by the sufficient conditions for Hamiltonicity.
Dirac's condition has been generalized to $k$ CISTs~\cite{hong20:_hamil,hasunuma16}, and Fleischner's condition has been generalized to $k$ CISTs~\cite{hong25:_compl}.
For more details, readers may refer to~\cite{cheng23}.

For CISTs in planar graphs, Hasunuma~\cite{hasunuma01} proved the following theorem.

\begin{thm}[\cite{hasunuma01}]
  \label{thm:hasu}
  There are two completely independent spanning trees in any 4-connected planar triangulation.
\end{thm}

Based on this theorem, Hasunuma~\cite{hasunuma01} conjectured that every $2k$-connected graph has $k$ CISTs.
However, P\'{e}terfalvi~\cite{peterfalvi12} disproved this conjecture by constructing a $k$-connected graph that does not have two CISTs for any $k \geq 2$.
In addition, P\'{e}terfalvi proved that the 4-connectivity of Theorem~\ref{thm:hasu} cannot be weakened; that is, there exists a 3-connected plane triangulation that does not contain two CISTs.
Wang and Liu~\cite{wang26} proved that there are infinitely many 4-connected planar graphs that do not have two CISTs.

Motivated by the above facts, we consider sufficient conditions for the existence of CISTs in 3-connected triangulated discs.
We first show that every 3-connected 2-outerplanar triangulated disc has two CISTs.
Then, in Section~\ref{sec:3op}, we provide sufficient conditions for a 3-connected 3-outerplanar graph to have two CISTs.
In Section~\ref{sec:4op}, we provide an example of a 3-connected 4-outerplanar triangulated disc that does not contain two CISTs.


\section{Preliminaries}
\label{sec:prelim}

Hasunuma~\cite{hasunuma01} introduced the notion of completely independent spanning trees and provided the following characterization.

\begin{thm}[\cite{hasunuma01}]
  \label{thm:ch01}
  Let $T_{1}, T_{2}, \dots, T_{k}$ be $k$ spanning trees of $G$.
  Then, $T_{1}, T_{2}, \dots, T_{k}$ are completely independent if and only if $T_{1}, T_{2}, \dots, T_{k}$ are edge-disjoint, and for any vertex $v$, there is at most one spanning tree $T_{i}$ such that $\deg_{T_{i}} v > 1$.
\end{thm}

Let $(V_{1}, V_{2}, \dots, V_{k})$ be a partition of the vertex set $V(G)$, and for $i \neq j$, $B(V_{i}, V_{j}, G)$ be the bipartite graph with bipartition $V_{i} \cup V_{j}$ and edge set $\{uv \mid uv \in E(G), u \in V_{i}, v \in V_{j}\}$.
We may use $B(V_{i}, V_{j})$ instead of $B(V_{i}, V_{j}, G)$ if $G$ is clear from context.
A partition $(V_{1}, V_{2}, \dots, V_{k})$ is called a \emph{CIST-partition} of $G$ if it satisfies the following two conditions:
\begin{itemize}
\item[(1)] for $i = 1, 2, \dots, k$, the induced subgraph $G[V_{i}]$ is connected, and
\item[(2)] for any $i \neq j$, the bipartite graph $B(V_{i}, V_{j})$ has no tree components, that is, every connected component $H$ of $B(V_{i}, V_{j})$ satisfies $|E(H)| \geq |V(H)|$.
\end{itemize}
The author proved the following characterization of CISTs.

\begin{thm}[\cite{araki14}]
  \label{thm:araki}
  A connected graph $G$ has $k$ completely independent spanning trees if and only if there exists a CIST-partition $(V_{1}, V_{2}, \dots, V_{k})$ of $G$.
\end{thm}

Let $G$ be a $k$-outerplanar graph, and let us fix a $k$-outerplanar embedding of $G$.
Define $L_{k}$ as the set of vertices on the outer face, and define $L_{i}$ for $1 \leq i \leq k-1$ recursively as the set of vertices on the outer face of the plane graph obtained by removing the vertices in $L_{i+1} \cup \dots \cup L_{k}$.
We call $L_{i}$ the \emph{$i$-th layer} of $G$.

When $G$ is 3-connected and its inner faces are triangles, the following properties are satisfied:

\begin{prop}
  \label{p:one}
  The induced subgraph $G[L_{k}]$ is a chordless cycle of $ G$.
\end{prop}

\begin{prop}
  \label{p:two}
  $G - L_{k}$ is connected.
\end{prop}

\begin{prop}
  \label{p:four}
  For $1 \leq i \leq k-1$ and for any vertex $u \in L_{i}$, there exists a vertex $w \in L_{i+1}$ such that $uw \in E(G)$.
\end{prop}

\begin{prop}
  \label{p:three}
  For any edge $e = uv$ on $G[L_{k}]$, there exists a vertex $w \in L_{k-1}$ such that $u,v,w$ form a triangle.
\end{prop}


\section{2-Outerplanar triangulated discs}
\label{sec:2op}

Let $G$ be a 3-connected 2-outerplanar triangulated disc.
By Properties~\ref{p:one} and~\ref{p:two}, $G[L_{2}]$ is a chordless
cycle and $G[L_{1}]$ is connected.

\begin{thm}
  \label{thm:2out}
  Every 3-connected 2-outerplanar triangulated disc $G$ has two completely independent spanning trees.
  Furthermore, $G$ has a CIST-partition $(V_{1}, V_{2})$ such that $|V_{1} \cap L_{2}| = 1$.
\end{thm}
\begin{proof}
  Let $G$ be a 3-connected 2-outerplanar triangulated disc and $L_{1}, L_{2}$ be the first and second layers of $G$, respectively.

  We consider two cases.

  \noindent
  (Case~1) $|L_{1}| = 1$.
  Let $L_{1} = \{u\}$ and $v \in L_{2}$ be a vertex adjacent to $u$.
  Define $V_{1} = \{u, v\}$ and $V_{2} = L_{2} \setminus \{v\}$.
  Clearly, $G[V_{1}]$ and $G[V_{2}]$ are connected.
  It is easy to see that $B(V_{1}, V_{2})$ is connected and not a tree.
  Hence, $(V_{1}, V_{2})$ is a CIST-partition of $G$ and $|V_{1} \cap L_{2}| = 1$.

  \begin{figure}[tbp]
    \centering
    \includegraphics{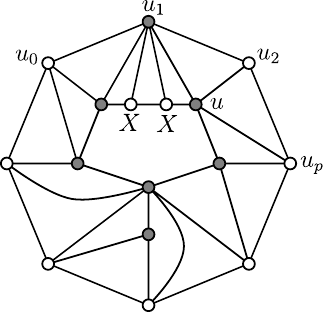}
    \caption{The situation of Case~2 of Theorem~\ref{thm:2out}.
      Some edges joining vertices of $L_{1}$ are omitted.
      The black and white vertices are members of $V_{1}$ and $V_{2}$, respectively.
      The vertices labeled by $X$ are members of $X$.}
    \label{fig:th31}
  \end{figure}

  \noindent
  (Case~2) $|L_{1}| \geq 2$.
  In this case, there exists a vertex $u \in L_{1}$ such that $|N_{G}(u) \cap L_{2}| \geq 2$.
  Let $N_{G}(u) \cap L_{2} = \{u_{1},u_{2},\dots,u_{p}\}$, $p \geq 2$, and $u_{1}, u_{2}, \dots, u_{p}$ be arranged clockwise on the chordless cycle $G[L_{2}]$.
  Let $X = \{v \mid v \in L_{1}, N_{G}(v) \cap L_{2} = \{u_{1}\}\}$.
  Define $V_{1} = \{u_{1}\} \cup (L_{1} \setminus X)$ and $V_{2} = (L_{2} \setminus \{u_{1}\}) \cup X$ (see Fig.~\ref{fig:th31}).

  By Properties~\ref{p:one} and \ref{p:two}, the subgraphs $G[V_{1}]$ and $G[V_{2}]$ are connected.
  Consider the bipartite graph $B = B(V_{1}, V_{2})$.
  To prove $(V_{1}, V_{2})$ is a CIST-partition, we show that $B$ is a connected subgraph with a cycle.

  Let $w$ be any vertex of $G$.
  We show that there exists a $u_{1}$--$w$ walk in $B$.
  Let $u_{0} \in L_{2}$ be the vertex immediately before $u_{1}$ in the clockwise direction.
  Since $u_{0}u_{1}, u_{1}u_{2} \in E(B)$ and $uu_{2},\dots,uu_{p} \in E(B)$, there is a $u_{1}$--$w$ walk in $B$ for $w \in \{u, u_{0}, u_{2}, \dots, u_{p}\}$.

  Consider the case where $w \in V_{2}$ and $w \notin \{u_{0}, u_{1}, \dots, u_{p}\}$.
  Let $P=(u_{p}=x_{0}, x_{1}, \dots, x_{q} = w)$ be a $u_{p}$--$w$ path on $G[V_{2}]$ such that $P$ does not pass through $u_{1}$.
  By Property~\ref{p:three}, for each edge $x_{i} x_{i+1}$, there exists $y_{i} \in V_{1} \cap L_{1}$ such that $x_{i}y_{i}, y_{i}x_{i+1} \in E(B)$.
  Hence we obtain a $u_{p}$--$w$ walk $W = (u_{p} = x_{0}, y_{0}, x_{1}, \dots, x_{q-1}, w_{q-1}, x_{q} = w)$.
  We obtain a $u_{1}$--$w$ walk in $B$ from $W$ and the path $(u_{1}, u_{2}, u, u_{p})$.

  Next, we suppose that $w \in L_{1}$.
  If $w \in X$, then $w$ is adjacent to $u_{1}$ in $B$.
  If $w \notin X$, by Property~\ref{p:four}, there exists a vertex $w' \in V_{2}$ such that $ww' \in E(B)$.
  By a similar argument, there is a $u_{1}$--$w'$ walk $W$ on $B$, and thus we obtain a $u_{1}$--$w$ walk on $B$ from $W$ and the edge $ww'$.

  From the above discussion, we see that $B$ is connected.
  The bipartite graph $B$ has a $u_{0}$--$u_{p}$ path $P$ that does not contain vertices $u_{1}, \dots, u_{p-1}$.
  We obtain a cycle of $B$ from $P$ and the path $(u_{0}, u_{1}, u_{2}, u, u_{p})$.

  From the above argument, $(V_{1}, V_{2})$ is a CIST-partition of $G$ and $|V_{1} \cap L_{2}| = 1$.
\end{proof}

Not every 2-connected 2-outerplanar triangulated disc has two CISTs.
For example, the graph $G$ shown in Fig.~\ref{fig:noCISTs} does not have two CISTs.
Assume, to the contrary, that the graph $G$ in Fig.~\ref{fig:noCISTs} has a CIST-partition $(V_{1}, V_{2})$.
By the definition of the CIST-partition, every vertex in $V_{1}$ is adjacent to some vertex in $V_{2}$, and every vertex in $V_{2}$ is adjacent to some vertex in $V_{1}$.
Thus, $u, v$ and $w$ cannot all be in $V_{1}$.
Without loss of generality, we may assume that $u, v \in V_{1}$ and $w \in V_{2}$.
Then, at least one of $x$ or $y$ is in $V_{2}$, but $G[V_{2}]$ is not connected, because there is no path between $w$ and $x$ or $y$ in $G[V_{2}]$.
This is a contradiction.

\begin{figure}[tbp]
  \centering
  \includegraphics{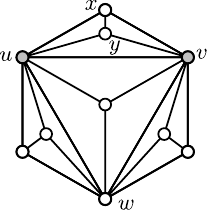}
  \caption{An example of a 2-connected 2-outerplanar triangulated disc $G$ that has no CISTs.}
  \label{fig:noCISTs}
\end{figure}

We provide a sufficient condition for a 2-connected 2-outerplanar triangulated disc to have two CISTs.
A 1-outerplanar triangulated disc (or maximal outerplanar graph) is \emph{striped} if every inner triangle shares an edge with the outer face.

\begin{lem}
  \label{lem:2-color}
  Let $G$ be a striped 1-outerplanar triangulated disc of $n \geq 3$ vertices.
  Then, the vertex set of $G$ has a partition $(U_{1}, U_{2})$ such that $U_{1} \neq \emptyset$ and $U_{2} \neq \emptyset$ and every inner triangle contains either exactly one vertex in $U_{1}$ or exactly one vertex in $U_{2}$.
  Furthermore, the vertices in $U_{1}$ are arranged consecutively on the outer face of $G$.
\end{lem}
\begin{proof}
  If $n=3$, then the lemma holds.
  For $n \geq 4$, it is known that $G$ has exactly two nonadjacent vertices of degree 2.
  Let $u_{1}, u_{2}, \dots, u_{n}$ be the vertices of $G$ that are arranged in the clockwise direction on the outer face of $G$.
  Let $\deg_{G} u_{1} = \deg_{G} u_{k} = 2$ for some $3 \leq k \leq n-1$.
  Define $U_{1} = \{u_{1}, u_{2}, \dots, u_{k-1}\}$ and $U_{2} = V(G) \setminus U_{1}$.
  The partition $(U_{1}, U_{2})$ has the desired property since any chord of $G$ joins a vertex in $U_{1}$ and $U_{2}$.
  See Fig.~\ref{fig:th33}.
\end{proof}

\begin{figure}[tbp]
  \centering
  \includegraphics{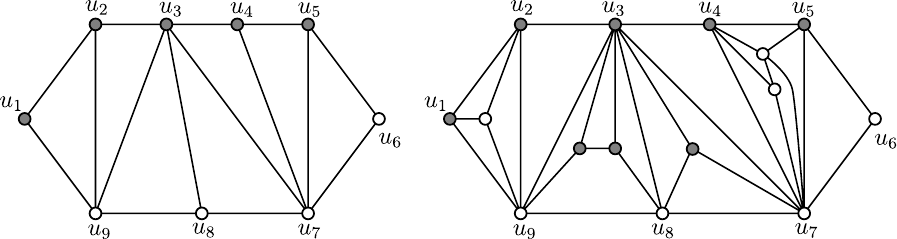}
  \caption{(Left) An example of a striped 1-outerplanar triangulated disc for Lemma~\ref{lem:2-color}.
    (Right) An example of a 2-connected 2-outerplanar triangulated disc for Theorem~\ref{thm:2-c-2-out}.
    The black and white vertices are members of $V_{1}$ and $V_{2}$, respectively.}
  \label{fig:th33}
\end{figure}

\begin{thm}
  \label{thm:2-c-2-out}
  Let $G$ be a 2-connected 2-outerplanar triangulated disc.
  If $G[L_{2}]$ is a striped 1-outerplanar triangulated disc and $G[L_{1}]$ is a forest, then $G$ has two completely independent spanning trees.
  Furthermore, $G$ has a CIST-partition $(V_{1}, V_{2})$ such that the vertices of $V_{1}$ are arranged consecutively in $G[L_{2}]$.
\end{thm}
\begin{proof}
  Assume that $G$ is a 2-connected 2-outerplanar triangulated disc, $G[L_{2}]$ is a striped 1-outerplanar triangulated disc, and $G[L_{1}]$ is a forest.
  By Lemma~\ref{lem:2-color}, $L_{2}$ has a partition $(U_{1}, U_{2})$ such that $U_{1} \neq \emptyset$ and $U_{2} \neq \emptyset$ and every triangle of $G[L_{2}]$ contains either exactly one vertex in $U_{1}$ or exactly one vertex in $U_{2}$.
  Furthermore, the vertices in $U_{1}$ and $U_{2}$ are arranged consecutively on the outer face of $G$.

  Since $G[L_{1}]$ is a forest and every inner face of $G[L_{2}]$ is a triangle, every inner triangle of $G[L_{2}]$ contains at most one component of $G[L_{1}]$.
  We construct a partition $(V_{1}, V_{2})$ of $V(G)$ as follows:
  \begin{enumerate}
  \item $V_{i}$ contains every vertex of $U_{i}$ for $i = 1, 2$.
  \item If a component $T$ of $G[L_{1}]$ is contained in a triangle of $G[L_{2}]$ and the triangle has exactly one vertex in $U_{i}$, then every vertex of $T$ is in $V_{i}$.
  \end{enumerate}
  We show that $(V_{1}, V_{2})$ is a CIST-partition of $G$.

  Since the vertices in $U_{1}$ and $U_{2}$ are arranged consecutively on the outer face of $G$, $G[U_{1}]$ and $G[U_{2}]$ are connected.
  Let $T$ be a component of $G[L_{1}]$ and $C$ be the triangle of $G[L_{2}]$ that contains $T$.
  Assume that the vertices of $T$ are members of $V_{i}$.
  Because every face is a triangle, there is at least one edge between a vertex of $T$ and a vertex of $U_{i} \cap V(C)$.
  Hence, $G[V_{i}]$ is connected for $i = 1, 2$.

  Let $B = B(V_{1}, V_{2})$.
  We show that $B$ is a connected bipartite graph with cycles.
  It is easy to see that $B[L_{2}]$ is a tree and hence connected.
  Let $T$ be a component of $G[L_{1}]$ and $C$ be the triangle of $G[L_{2}]$ that contains $T$.
  Let $V(C) = \{x, y, z\}$, $x \in V_{i}$ and $y, z \notin V_{i}$.
  Therefore, the vertices of $T$ are members of $V_{i}$.
  Since $G$ is a triangulated disc, there is a vertex $w$ of $T$ such that $y, z, w$ form a triangle.
  Hence, the edges $xy, yw, wz$ and $zx$ form a cycle of $B$.

  Therefore, $(V_{1}, V_{2})$ is a CIST-partition of $G$.
\end{proof}


\section{3-Outerplanar graphs}
\label{sec:3op}

Suppose that $G$ is a $k$-outerplanar triangulated disc, and $L_{1}, L_{2}, \dots, L_{k}$ are its $k$ layers.
For $1 \leq i \leq k$, let $G_{i}$ be the subgraph induced by $L_{1} \cup L_{2} \cup \dots \cup L_{i}$.
From the definition, $G_{i}$ is an $i$-outerplanar triangulated disc and $G_{k} = G$.

\begin{lem}
  \label{lem:extend}
  Let $2 \leq i \leq k-1$.
  Suppose that $G_{i}$ is 2-connected and has a CIST-partition $(V_{1}, V_{2})$ and $G[L_{i+1}]$ is a chordless cycle.
  If $1 \leq |L_{i} \cap V_{1}| < |L_{i}|$ and the vertices in $L_{i} \cap V_{1}$ are arranged consecutively on $G[L_{i}]$, then $G_{i+1}$ has a CIST-partition $(V'_{1}, V'_{2})$ such that $1 \leq |L_{i+1} \cap V'_{2}| < |L_{i+1}|$ and the vertices in $L_{i+1} \cap V'_{2}$ are arranged consecutively on $G[L_{i+1}]$.
\end{lem}
\begin{proof}
  Let $L_{i} = \{u_{1}, u_{2}, \dots, u_{p}\}$, and assume $u_{1}, u_{2}, \dots, u_{p}$ are arranged in the clockwise direction on $G[L_{i}]$.
  Assume that $(V_{1}, V_{2})$ is a CIST-partition of $G_{i}$ such that $U = L_{i} \cap V_{1} = \{u_{1}, u_{2}, \dots, u_{c}\}$ for some $1 \leq c < p$.
  Define $N_{U} = N_{G}(U) \cap L_{i+1}$.
  Assume that $G[L_{i+1}]$ is a chordless cycle of $G$ and $L_{i+1} = \{v_{1}, v_{2}, \dots, v_{q}\}$ and $v_{1}, v_{2}, \dots, v_{q}$ are arranged in the clockwise direction on $G[L_{i+1}]$.
  Without loss of generality, we may assume that $N_{U} = \{v_{1}, v_{2}, \dots, v_{d}\}$.

  \vspace{1em}
  \noindent
  Case 1. $|N_{U}| = 1$.

  First, we consider the case of $N_{U} = \{v_{1}\}$.
  Define $V'_{1} = (L_{i+1} \setminus \{v_{2}\}) \cup V_{1}$ and $V'_{2} = \{v_{2}\} \cup V_{2}$.
  This is illustrated in Fig.~\ref{fig:lem41}(a).
  We show that $(V'_{1}, V'_{2})$ is a CIST-partition of $G_{i+1}$.
  It is easy to see that $G_{i+1}[V'_{1}]$ is connected.
  By Property~\ref{p:three}, vertex $v_{2} \in V'_{2}$ is adjacent to some vertex $u' \in L_{i}$.
  Since $N_{U} = \{v_{1}\}$, it holds $u' \notin U$ and hence $u' \in V'_{2}$.
  Thus, $G_{i+1}[V'_{2}]$ is connected.

  Next, we show that $B = B(V'_{1}, V'_{2}, G_{i+1})$ does not have a tree component.
  Let
  \begin{align*}
    X &= \{u_{j} \mid 2 \leq j \leq c-1 \text{ and  } N_{G}(u_{j}) \cap
        L_{i+1} = \{v_{1}\}\} \\
      &\cup
    \{u_{j} \mid c+1 \leq j \leq q \text{ and } N_{G}(u_{j}) \cap L_{i+1} = \{v_{2}\}\}.
  \end{align*}
  Since $(V_{1}, V_{2})$ is a CIST-partition of $G_{i}$, the components of $B(V_{1}, V_{2}, G_{i})$ that contain vertices of $X$ have cycles in them.
  Let $L'_{i} = L_{i} \setminus X$.
  Note that every vertex in $L'_{i}$ is adjacent to some vertex in $V'_{1} \cap L_{i+1}$.
  Define $B_{i} = B[L'_{i} \cup L_{i+1}]$.
  We show that $B_{i}$ is a connected graph with a cycle.
  Since $N_{U} = \{v_{1}\}$ and $G_{i+1}$ is a triangulated disc, $v_{1}$ is adjacent to $u_{c+1}$ and $u_{p}$.

  \begin{figure}[tbp]
    \centering
    \includegraphics{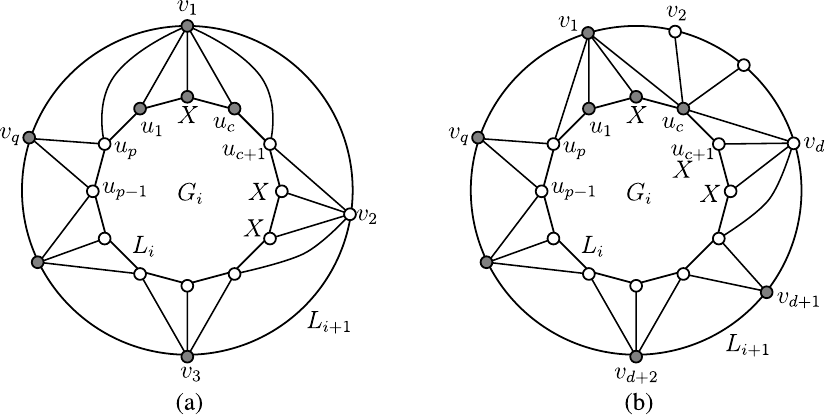}
    \caption{The situation of (a) Case~1 and (b) Case~2 of Lemma~\ref{lem:extend}.
      The black and white vertices are members of $V'_{1}$ and $V'_{2}$, respectively.
      The vertices labeled by $X$ are members of $X$.}
    \label{fig:lem41}
  \end{figure}

  First, we prove that $B_{i}$ is connected by showing that there is a $v_{1}$--$y$ walk in $B_{i}$ for any $y \in L'_{i} \cup L_{i+1}$.
  Let $y = v_{j} \in L_{i+1}$.
  Since $v_{1}v_{2}, v_{2}v_{3} \in E(B_{i})$, there exists a $v_{1}$--$v_{j}$ walk in $B_{i}$ for $j = 2, 3$.
  Let $4 \leq j \leq q$ and $P=(v_{3}, v_{4}, \dots, v_{j})$ be a $v_{3}$--$v_{j}$ path on $G[L_{i+1}]$.
  By Property~\ref{p:three}, for each edge $v_{t} v_{t+1}$ for $t \geq 3$, there exists $w_{t} \in V'_{2} \cap L'_{i}$ such that $v_{t}w_{t}, v_{t+1}w_{t} \in E(B_{i})$.
  We obtain a $v_{3}$--$v_{j}$ walk $W = (v_{3}, w_{3}, v_{4}, \dots, v_{j-1}, w_{j-1}, v_{j})$.
  Thus we obtain a $v_{1}$--$v_{j}$ walk in $B_{i}$ from $W$ and edges
  $v_{1}v_{2}, v_{2}v_{3}$.
  Let $y \in L'_{i}$.
  If $y = u_{1}$, then $(v_{1}, u_{p}, u_{1} = y)$ is a $v_{1}$--$y$ path in $B_{i}$.
  If $y = u_{c}$, then $(v_{1}, u_{c+1}, u_{c} = y)$ is a $v_{1}$--$y$ path in $B_{i}$.
  If $y = u_{j} \in V'_{2}$, then there is a $v_{1}$--$v_{j'}$ walk in $B_{i}$ where $v_{j'} \in V'_{1} \cap L_{i+1}$ is a vertex adjacent to $y$.
  Thus, we obtain a $v_{1}$--$y$ walk in $B_{i}$, and hence $B_{i}$ is connected.

  Next, we show that $B_{i}$ has a cycle.
  If $c = |N_{U}| = 1$, then $(u_{1}, u_{2}, v_{1}, u_{p}, u_{1})$ is a cycle in $B_{i}$.
  Then, suppose that $c \geq 2$.
  From the above discussion, subgraph $B_{i}$ has a $v_{3}$--$v_{q}$ path $P$ that does not contain $v_{1}$ and $v_{2}$.
  We obtain a cycle of $B_{i}$ from $P$ and the path $(v_{q}, u_{p}, v_{1}, v_{2}, v_{3})$.

  Thus, the component of $B$ that contains the vertices of $B_{i}$ is a connected graph with a cycle.
  Hence, $(V'_{1}, V'_{2})$ is a CIST-partition of $G_{i+1}$, and $|V'_{2} \cap L_{i+1}| = 1$.

  \vspace{1em}
  \noindent
  Case 2. $|N_{U}| \geq 2$.

  Next, we consider the case of $N_{U} = \{v_{1}, v_{2}, \dots, v_{d}\}$ for $d \geq 2$.
  Define $V'_{1} = (L_{i+1} \setminus \{v_{2}, \dots, v_{d}\}) \cup V_{1}$ and $V'_{2} = \{v_{2}, \dots, v_{d}\} \cup V_{2}$.
  This is illustrated in Fig.~\ref{fig:lem41}(b).
  We show that $(V'_{1}, V'_{2})$ is a CIST-partition of $G_{i+1}$.
  Since $u_{1}v_{1}, u_{c+1}v_{d} \in E(G)$, $G_{i+1}[V'_{1}]$ and $G_{i+1}[V'_{2}]$ are connected.

  Next, we show that $B = B(V'_{1}, V'_{2}, G_{i+1})$ does not have a tree component.
  Let
  \begin{align*}
    X &= \{u_{i} \mid 2 \leq i \leq c-1 \text{ and  } N_{G}(u_{i}) \cap
        L_{i+1} = \{v_{1}\}\} \\
      &\cup
    \{u_{i} \mid c+1 \leq i \leq q \text{ and } N_{G}(u_{i}) \cap L_{i+1} = \{v_{d}\}\}.
  \end{align*}
  Since $(V_{1}, V_{2})$ is a CIST-partition of $G_{i}$, the components of $B(V_{1}, V_{2}, G_{i})$ that contain vertices of $X$ have cycles in them.
  Let $L'_{i} = L_{i} \setminus X$.
  Note that every vertex in $L'_{i} \cap V'_{1}$ is adjacent to some vertex in $V'_{2}$, and every vertex in $L'_{i} \cap V'_{2}$ is adjacent to some vertex in $V'_{1}$.
  Define $B_{i} = B[L'_{i} \cup L_{i+1}]$.
  We show that $B_{i}$ is a connected graph with a cycle.
  Since $N_{U} = \{v_{1}, v_{2}, \dots, v_{d}\}$ and $G_{i+1}$ is a triangulated disc, $v_{1}$ is adjacent to $u_{p}$ and $v_{d}$ is adjacent to $u_{c+1}$.

  First, we prove that $B_{i}$ is connected by showing that there is a $v_{1}$--$y$ walk in $B_{i}$ for any $y \in L'_{i} \cup L_{i+1}$.
  Let $y = v_{j} \in L_{i+1}$.
  When $j=2$, $(v_{1}, v_{2})$ is a $v_{1}$--$v_{2}$ walk.
  For $3 \leq j \leq d$, let $P=(v_{2}, v_{3}, \dots, v_{j})$ be a $v_{2}$--$v_{j}$ path on $G[L_{i+1}]$.
  By Property~\ref{p:three}, for each edge $v_{t} v_{t+1}$, there exists $w_{t} \in L'_{i}$ such that $v_{t}w_{t}, v_{t+1}w_{t} \in E(B_{i})$.
  Thus we obtain a $v_{1}$--$v_{j}$ walk $W = (v_{1}, v_{2}, w_{2}, v_{3}, \dots, v_{j-1}, w_{j-1}, v_{j})$ in $B_{i}$.
  For $y = v_{d+1}$, we obtain a $v_{1}$--$v_{d+1}$ walk from the $v_{1}$--$v_{d}$ walk and the edge $v_{d}v_{d+1}$.
  For $d+2 \leq j \leq q$, we can construct a $v_{d+1}$--$v_{j}$ walk $W$ in $B_{i}$ in a manner similar to Case~1.
  Thus we obtain a $v_{1}$--$v_{j}$ walk from the $v_{1}$--$v_{d+1}$ walk and the $v_{d+1}$--$v_{j}$ walk.

  Let $y \in L'_{i}$.
  If $y = u_{1}$, then $(v_{1}, u_{p}, u_{1})$ is a $v_{1}$--$u_{1}$ path in $B_{i}$.
  If $y \neq u_{1}$, then there is a $v_{1}$--$v_{j}$ walk in $B_{i}$ where $v_{j}$ is a vertex adjacent to $y$.
  Thus, we obtain a $v_{1}$--$y$ walk in $B_{i}$.
  Hence, $B_{i}$ is connected.

  Next, we show that $B_{i}$ has a cycle.
  From the above discussion, the subgraph $B_{i}$ has a $v_{1}$--$v_{d}$ path $P$ such that $V(P) \subseteq U \cup N_{U}$, and $v_{d}$--$v_{1}$ path $P_{1}$ such that $V(P_{1}) \cap (U \cup N_{U}) = \{v_{1}, v_{d}\}$.
  We obtain a cycle of $B_{i}$ from $P$ and $P_{1}$.

  Thus, the component of $B$ that contains the vertices of $B_{i}$ is a connected graph with a cycle.
  Hence, $(V'_{1}, V'_{2})$ is a CIST-partition of $G_{i+1}$, and it holds $1 \leq |V'_{2} \cap L_{i+1}| < q$ and the vertices in $V'_{2}$ are arranged consecutively in $G[L_{i+1}]$.

  We obtained the desired results for both cases.
\end{proof}

From Theorem~\ref{thm:2out} and Lemma~\ref{lem:extend}, we obtain the following theorems.

\begin{thm}
  \label{thm:k-out}
  Let $G$ be a $k$-outerplanar triangulated disc.
  If $G_{2}$ is 3-connected and $G[L_{i}]$ are chordless cycles for every $3 \leq i \leq k$, then $G$ has two completely independent spanning trees.
\end{thm}

\begin{thm}
  \label{thm:32-out}
  Let $G$ be a 3-outerplanar triangulated disc.
  If $G_{2}$ is 3-connected, then $G$ has two completely independent spanning trees.
\end{thm}
\begin{proof}
  By Theorem~\ref{thm:2out}, $G_{2}$ has a CIST-partition
  $(V_{1}, V_{2})$ such that $|V_{1} \cap L_{2}| = 1$.
  Since $G$ is 3-connected, $G[L_{3}]$ is a chordless cycle of $ G$.
  By Theorem~\ref{thm:k-out}, $G$ has two CISTs.
\end{proof}

\begin{thm}
  \label{thm:3-out-2}
  Let $G$ be a 3-connected 3-outerplanar triangulated disc.
  If $G[L_{2}]$ is striped 1-outerplanar and $G[L_{1}]$ is a forest, then $G$ has two completely independent spanning trees.
\end{thm}
\begin{proof}
  By Theorem~\ref{thm:2-c-2-out}, $G_{2}$ has a CIST-partition
  $(V_{1}, V_{2})$ such that the vertices in $V_{1}$ are arranged consecutively in the clockwise direction on $G[L_{2}]$.
  Since $G$ is 3-connected, $G[L_{3}]$ is a chordless cycle of $G$.
  By Theorem~\ref{thm:k-out}, $G$ has two CISTs.
\end{proof}

The sufficient conditions presented so far require $G_{2}$ to have two CISTs.
However, there exist examples where $G_{2}$ does not have CISTs but $G$ has two CISTs.
The graph in Fig.~\ref{fig:CIST-3-3} shows a 3-connected 3-outerplanar triangulated disc $G$ with a CIST-partition.
Subgraph $G_{2} = G[L_{1} \cup L_{2}]$ is isomorphic to the graph in Fig.~\ref{fig:noCISTs}; thus, $G$ is an example of a triangulated disc that is 3-connected, and $G_{2}$ does not have a CIST-partition.

\begin{figure}[tbp]
  \centering
  \includegraphics{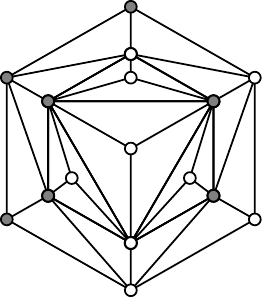}
  \caption{A 3-connected 3-outerplanar graph $G$ that has a CIST-partition $(V_{1}, V_{2})$.
    The sets of black and white vertices are members of $V_{1}$ and $V_{2}$, respectively.
    The subgraph $G_{2} = G[L_{1} \cup L_{2}]$ is isomorphic to the graph shown in Fig.~\ref{fig:noCISTs}.
  }
  \label{fig:CIST-3-3}
\end{figure}


\section{4-Outerplanar graphs}
\label{sec:4op}

In this section, we show that there exists a 3-connected 4-outerplanar triangulated disc that does not have two CISTs.
P\'{e}terfalvi~\cite{peterfalvi12} showed that a 3-connected plane triangulation with no CISTs can be constructed from a 3-connected 3-regular non-Hamiltonian planar graph.
Let $G$ be a 3-connected, 3-regular, non-Hamiltonian planar graph, and let $G^{*}$ be the dual of $G$.
Then, $G^{*}$ is a 3-connected plane triangulation.
Let us add a new vertex to every face (including the outer face) of $G^{*}$ and connect it to each vertex of the face.
The resulting graph, denoted by $D_{G}$, is also 3-connected and triangulated.
In~\cite{peterfalvi12}, it was shown that the graph $D_{G}$ does not have two CISTs.

For example, the Tutte graph is 3-connected, 3-regular, and non-Hamiltonian.
Figure~\ref{fig:tutte} shows the Tutte graph $G$, and its dual is shown in Fig.~\ref{fig:tutte_dual}.
Then, $D_{G}$ is the graph shown in Fig.~\ref{fig:dual_tutte_d}, and we can see that the embedding is 4-outerplanar. 
Hence, we obtain the following theorem.

\begin{thm}
  \label{thm:4out}
  There exists a 3-connected triangulated disc with a 4-outerplanar embedding that does not have two completely independent spanning trees.
\end{thm}

\begin{figure}[tbp]
  \centering
  \includegraphics[scale=0.6]{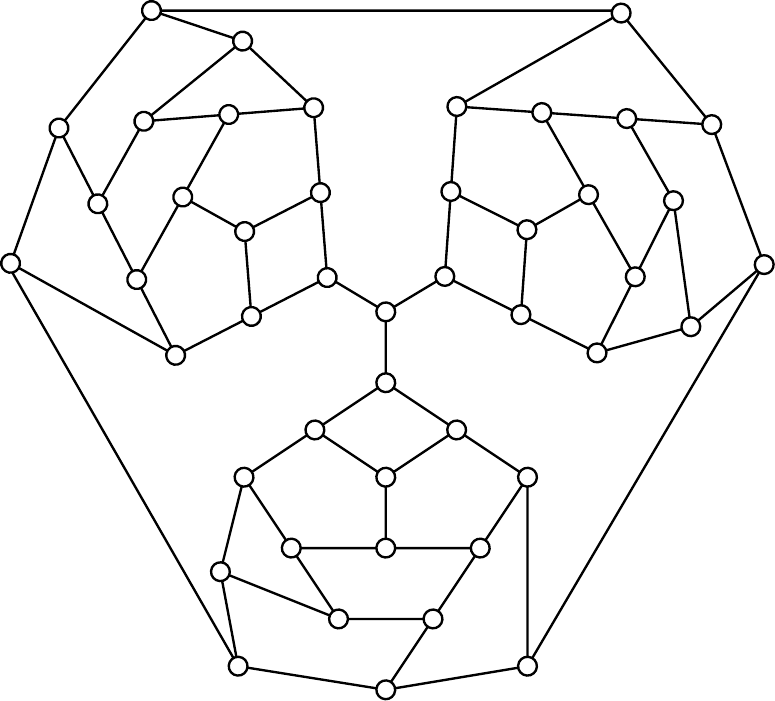}
  \caption{Tutte graph $G$.}
  \label{fig:tutte}
\end{figure}

\begin{figure}[tbp]
  \centering
  \includegraphics[scale=0.7]{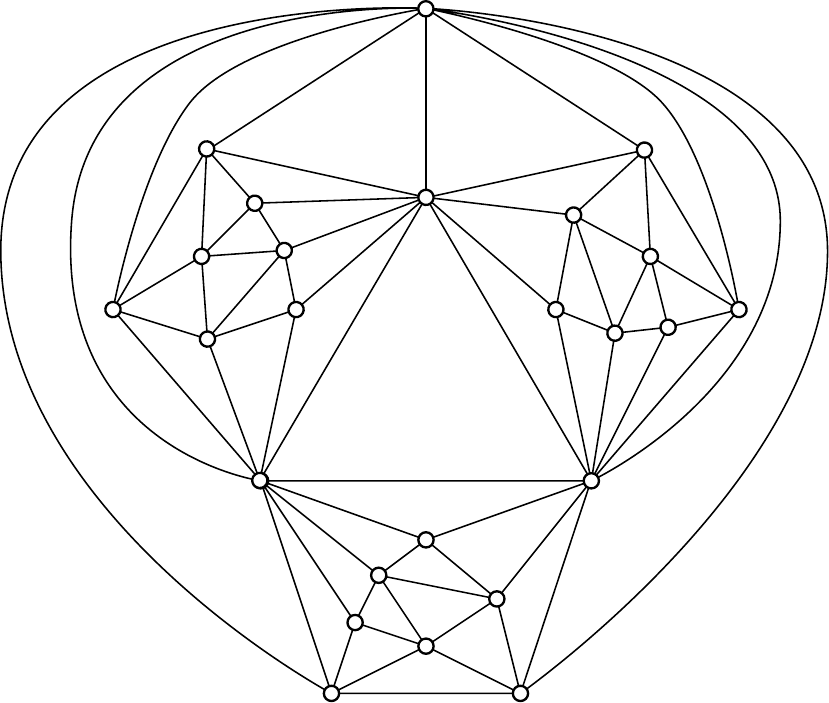}
  \caption{The dual of Tutte graph $G^{*}$.}
  \label{fig:tutte_dual}
\end{figure}

\begin{figure}[tbp]
  \centering
  \includegraphics[width=1.0\textwidth]{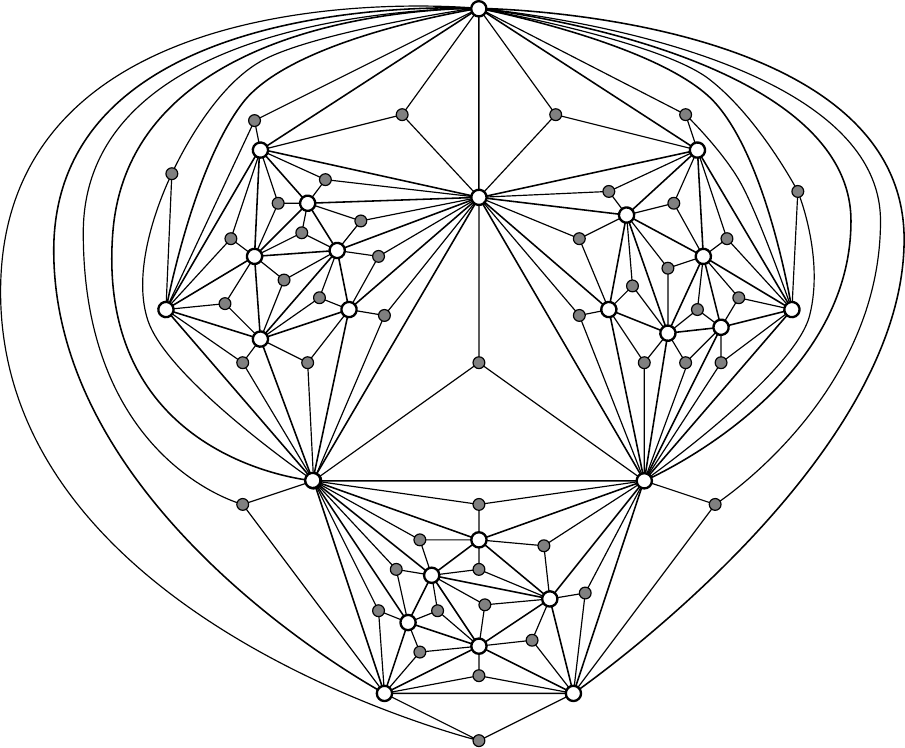}
  \caption{The graph $D_{G}$ that is made from the Tutte graph.
    The gray vertices are the vertices added to the dual $G^{*}$ to construct $D_{G}$.
    This graph does not have two completely independent spanning trees.
    We can see that this embedding is 4-outerplanar.}
  \label{fig:dual_tutte_d}
\end{figure}


\section{Conclusion}
\label{sec:conclusion}

Hasunuma~\cite{hasunuma02} proved that every 4-connected plane triangulation has two CISTs.
Later, P\'{e}terfalvi~\cite{peterfalvi12} pointed out that a 3-connected plane triangulation with no CISTs exists.

In this paper, we investigate the existence of CISTs in 3-connected $k$-outerplanar triangulated discs.
We showed some sufficient conditions for a 3-connected $k$-outerplanar graph to have two CISTs for $k = 2, 3$, and an example of a 3-connected 4-outerplanar triangulation with no CISTs.

The results of this study suggest that several problems remain regarding completely independent spanning trees and the outer-planarity of planar graphs.
We conclude this paper by proposing two conjectures:

\begin{conj}
  Every 3-connected 3-outerplanar triangulated disc has two completely independent spanning trees.
\end{conj}

\begin{conj}
  For any $k \geq 4$, there exists a 3-connected $k$-outerplanar triangulated disc with no completely independent spanning trees.
\end{conj}

\end{document}